\documentclass{amsart}
\usepackage[top=0.75in]{geometry}
\usepackage{amsmath,amsthm,amssymb, float}
\usepackage{enumerate, xcolor}
\usepackage{comment, hyperref}
\usepackage{todonotes}
\usepackage{tabularray}
\usepackage{tikz,pgffor}

\newtheorem{thm}{Theorem}[section]

\newtheorem{prop}[thm]{Proposition}

\newtheorem*{cor*}{Corollary}
\newtheorem{conj}{Conjecture}

\newtheorem*{question*}{Question}

\theoremstyle{definition}
\newtheorem{defi}[thm]{Definition}

\title{The Penults of Tak: Adventures in impartial, normal-play, positional games}

\author[B. Alexeev]{Boris Alexeev}
\address{}
\email{boris.alexeev@gmail.com}
\author[P. Ellis]{Paul Ellis}
\address{Department of Mathematics, Rutgers University,
110 Frelinghuysen Road,
Piscataway, NJ,
08854,
USA}
\email{paulellis@paulellis.org}
\author[M. Richter]{Michael Richter}
\address{Baruch College, Zicklin School of Business, Department of Economics, 55 Lexington Avenue, New York, NY 10010, USA}
\address{Royal Holloway, University of London, Department of Economics, 1 Egham Hill, Egham, TW20 1EX, United Kingdom}
\email{michael.dan.richter@gmail.com}
\author[T. A. Thanatipanonda]{Thotsaporn Aek Thanatipanonda}
\address{Science Division, Mahidol University International College,
999 Phutthamonthon Sai 4 Rd, Salaya, Phutthamonthon District,
Nakhon Pathom,
73170,Thailand}
\email{thotsaporn@gmail.com}

\begin{document}

\maketitle

\begin{abstract}
    For normal play, impartial games, we define \emph{penults} as those positions in which every option results in an immediate win for the other player.  We explore the number of tokens in penults of two positional games, \textsc{Impartial Tic} and \textsc{Impartial Tak}.  We obtain a complete classification in the former case,.  We then explore winning strategies and further directions.
\end{abstract}

\section{Penults}

You are teaching \textsc{chess} to a child.  You explain that the goal is to capture the king.  After several plays, you then explain that, really, we don't capture the king, but the game ends two moves prior in a special position called \textit{checkmate}.

You are presenting an introduction to combinatorial games in a math circle.  As usual, you start with $\textsc{subtract}\{1,2,3\}$, but you call it `Take-Away $1,2,3$.'  The participants are working in pairs, and you walk around the room, asking what they have discovered so far. Several pairs excitedly report that they have solved the game:  ``You just need to get it down to $4$!''

Two children are playing classical $\textsc{tic-tac-toe}$.  The first player has created a fork.  The second player puts down their pencil and says, ``You win.''

In each of these scenarios, it is intuitive to the players that the game has already ended just before the second-to-last move.  In these positions,
\[\text{`wherever I go, you win.'}\]  
We name these positions by leaving out the last two syllables of `penultimate', just as the players leave out the last two moves.

\begin{defi}\label{definition of penult}
    Let $\mathcal{G}$ be an impartial game position.  
    \begin{itemize}
        \item If $\mathcal{G}$ has no options, it is \emph{terminal}.  (`Game over!')
        \item If $\mathcal{G}$ is not terminal, but it has a terminal option, it is an \emph{ult}.  (`I win!')
        \item If $\mathcal{G}$ is not terminal, and all of its options are ults, it is a \emph{penult}.  (`Wherever I go, you win!')
    \end{itemize}
\end{defi}

In many treatments of $\textsc{dots and boxes}$ (see, for example, Chapter 1.7 of \cite{LIP}), the analysis begins once the board is `filled up' and any further moves would create a three-sided box.  These filled up positions are thus of some interest.  If we modify the game so that it ends once any box has been completed, then these `filled up' positions become the penults of the abbreviated game $\textsc{d\&b}$.


In the board game \cite{Tak}, players take turns placing and moving stones on a rectangular board, with the goal of creating a road from one edge of the board to the opposite edge.  Roads need not be straight lines, but only orthogonal connections count in the win condition.

In early 2023, Daniel Hodgins introduced the game \textsc{Impartial Tak} to Gord's Problem Incubator, an online discussion group.  In this game, players take turns placing tokens on the empty spaces of a square board.  The first player to create an orthogonal path from any edge to its opposite edge wins.  The only penult on a $2\times 2$ board is an empty one, and, up to symmetries of the square, there are only $2$ penults on a $3\times 3$ board, shown in Figure \ref{fig:3x3 example}.  
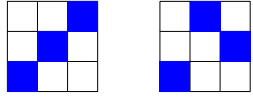
\begin{figure}[h]
    \centering
    \begin{tikzpicture}[scale=0.4] 
    \draw[step=1cm,black,very thin] (0,0) grid (3,3);
    \fill[blue] (0,0) rectangle (1,1);
    \fill[blue] (1,1) rectangle (2,2);
    \fill[blue] (2,2) rectangle (3,3);
    \end{tikzpicture}
    \qquad
    \begin{tikzpicture}[scale=0.4] 
    \draw[step=1cm,black,very thin] (0,0) grid (3,3);
    \fill[blue] (0,0) rectangle (1,1);
    \fill[blue] (1,2) rectangle (2,3);
    \fill[blue] (2,1) rectangle (3,2);
    \end{tikzpicture}
    \caption{All $3\times 3$ penults for \textsc{impartial tak}}
    \label{fig:3x3 example}
\end{figure}
Notice that the parities of these penults show that second player wins $2\times 2$ and first player wins $3\times 3$.
During the course of that discussion, several games were played on a $4\times4$ board. 
To some surprise, it was discovered that penults might have $6$, $7$, or $8$ tokens, as in Figure \ref{fig:4x4 example}, and thus the winning strategy is much less clear.
\begin{figure}[h]
    \centering
    \begin{tikzpicture}[scale=0.4] 
    \draw[step=1cm,black,very thin] (0,0) grid (4,4);
    \fill[blue] (0,2) rectangle (1,4);
    \fill[blue] (2,2) rectangle (3,4);
    \fill[blue] (1,0) rectangle (2,1);
    \fill[blue] (3,0) rectangle (4,1);
    \end{tikzpicture}
    \qquad
    \begin{tikzpicture}[scale=0.4] 
    \draw[step=1cm,black,very thin] (0,0) grid (4,4);
    \fill[blue] (0,2) rectangle (2,4);
    \fill[blue] (2,0) rectangle (4,1);
    \fill[blue] (3,1) rectangle (4,2);
    \end{tikzpicture}
    \qquad
    \begin{tikzpicture}[scale=0.4] 
    \draw[step=1cm,black,very thin] (0,0) grid (4,4);
    \fill[blue] (1,1) rectangle (3,3);
    \fill[blue] (0,0) rectangle (1,1);
    \fill[blue] (3,3) rectangle (4,4);
    \fill[blue] (0,3) rectangle (1,4);
    \fill[blue] (3,0) rectangle (4,1);
    \end{tikzpicture}
    \caption{Some $4\times 4$ penults for \textsc{impartial tak} with $6$, $7$, and $8$ tokens}
    \label{fig:4x4 example}
\end{figure}
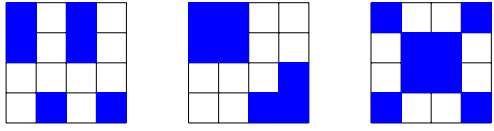
Hence the focus of this paper:

\begin{question*}\label{question: number of tokens}
    How many tokens are in the penults of \textsc{impartial tak}?
\end{question*}

Notice that both \textsc{impartial tak} and \textsc{d\&b} are positional games.  While we defined penults more generally, these are the games for which they are the most interesting.
In Section \ref{section - impartial tic}, we first turn our attention to another positional game, \textsc{impartial tic}.  In this game, players again take turns filling in an $n\times n$ grid with identical tokens, but the play ends when any row or column is completed.
In Section \ref{section - impartial tak} we explore the above question.  In Section \ref{section - winning strategies} we explore winning strategies for both games.

\section{The penults of \textsc{Impartial Tic}}\label{section - impartial tic}

When exploring the penults of \textsc{impartial tic}, we will find it helpful to instead draw penults for the dual game.  In this game, we begin with a board full of tokens, players take turn removing tokens from the board, and the game ends when any row or column is empty.  A penult then becomes a position satisfying
\begin{itemize}
\item  Each row and column has at least two tokens, and
\item  No token is in both a row and column with at least $3$ tokens.
\end{itemize}
Figure \ref{fig:canonical penults for tic} shows examples of the penults
\begin{align*}
\mathcal{A}_n &\text{ which has }2n\text{ tokens, defined for }n\geq 3\\
\mathcal{B}_n &\text{ which has }3(n-1)\text{ tokens, defined for }n\geq 3\\
\mathcal{C}_n &\text{ which has }4(n-2)\text{ tokens, defined for }n\geq 4
\end{align*}
\begin{figure}[h]
    \centering
    \begin{tikzpicture}[scale=0.4] 
    \draw[step=1cm,black,very thin] (0,0) grid (5,5);
    \fill[blue] (0,3) rectangle (1,5);
    \fill[blue] (1,2) rectangle (2,4);
    \fill[blue] (2,1) rectangle (3,3);
    \fill[blue] (3,0) rectangle (4,2);
    \fill[blue] (4,0) rectangle (5,1);
    \fill[blue] (4,4) rectangle (5,5);
    \end{tikzpicture}
    \qquad
    \begin{tikzpicture}[scale=0.4] 
    \draw[step=1cm,black,very thin] (0,0) grid (5,5);
    \fill[blue] (0,0) rectangle (1,4);
    \fill[blue] (1,4) rectangle (5,5);
    \fill[blue] (1,3) rectangle (2,4);
    \fill[blue] (2,2) rectangle (3,3);
    \fill[blue] (3,1) rectangle (4,2);
    \fill[blue] (4,0) rectangle (5,1);
    \end{tikzpicture}
    \qquad
    \begin{tikzpicture}[scale=0.4] 
    \draw[step=1cm,black,very thin] (0,0) grid (5,5);
    \fill[blue] (0,0) rectangle (2,3);
    \fill[blue] (2,3) rectangle (5,5);
    \end{tikzpicture}
    \caption{The penults $\mathcal{A}_5$, $\mathcal{B}_5$, and $\mathcal{C}_5$.}
\label{fig:canonical penults for tic}
\end{figure}
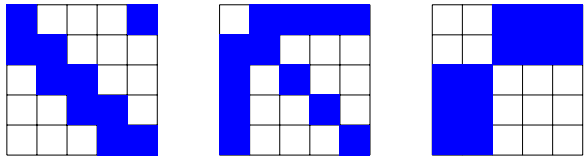
Since any penult must have at least $2$ tokens in each row, the lower bound of $2n$ is sharp.  In fact, the upper bound in these examples is sharp as well, and we can construct penults for each intermediate value.

\begin{prop}
If $n\geq 5$, the number of tokens in a penult is at most $4(n-2)$.  If $n=4$, then there are at most $9$ tokens
\end{prop}

\begin{proof}
Suppose we have a penult on an $n\times n$ board. Assume for a contradiction 
that either $n=4$ and there are at least $10$ tokens, or $n\geq 5$ and there are at least $4n-7$ tokens.  Call a token an $R$-token if it is in a row with at least $3$ tokens.  We claim that over half of the tokens must be $R$-tokens.  Since the position is a penult, any row containing a non-$R$-token must have exactly $2$ tokens in it.

In the case $n=4$, if at least $5$ are not $R$-tokens, then these must be in at least $3$ of the rows.  So we must have $3$ rows with exactly $2$ tokens each, and $1$ row with $4$ tokens.  In this case, one of the columns must have at least $3$ tokens, so the token in this column that is also in the filled row shows that this position is not a penult.

In the case $n\geq 5$, if at least $2n-3$ are not $R$-tokens, then these must be in at least $n-1$ of the rows. This means there are $2n-2$ tokens in these rows, leaving only one row for the remaining $2n-4  (>n)$ tokens, a contradiction. 

Similarly, at least half of the tokens must be in a column with at least $3$ tokens.  So by the pigeonhole principle, there is a token that is in both a row and a column with at least $3$ tokens, contradicting that the position is a penult.
\end{proof}
Thus for $n=4$, $\mathcal{A}_4$ and $\mathcal{B}_4$ show the full range of values.
\begin{prop}
If $n\geq 5$, $2n\leq x\leq 4(n-2)$, we can construct a penult on an $n\times n$ board with $x$ tokens.
\end{prop}

\begin{proof}
For $n\geq 6$, $3\leq k\leq n-3$, the otherwise empty $n\times n$ board with $\mathcal{A}_k$ in the upper left corner and $\mathcal{B}_{n-k}$ in the lower right corner is a penult.  This penult has $2k+3(n-k-1)=3n-k-3$ tokens.  These penults thus have all possible number of tokens in the interval $[2n,3n-6]$.

For $n\geq 7$, $3\leq k\leq n-4$, the otherwise empty $n\times n$ board with $\mathcal{B}_k$ in the upper left corner and $\mathcal{C}_{n-k}$ in the lower right corner is a penult.  This penult has $3(k-1)+4(n-k-2)=4n-k-11$ tokens.  These penults thus have all possible number of tokens in the interval $[3n-7,4n-14]$.

For $9\leq m\leq 13$, Figure \ref{fig:spradic penults for tic} shows examples of the penults $\mathcal{D}_{n,m}$, which have $4n-m$ tokens respectively.  To show this note
\begin{align*}
    \mathcal{D}_{n,9}&\text{ has }4(n-3)+3=4n-9\text{ tokens}\\
    \mathcal{D}_{n,10}&\text{ has }2(n-4)+2(n-3)+4=4n-10\text{ tokens}\\
    \mathcal{D}_{n,11}&\text{ has }4(n-4)+5=4n-11\text{ tokens}\\
    \mathcal{D}_{n,12}&\text{ has }4(n-4)+4=4n-12\text{ tokens}\\
    \mathcal{D}_{n,13}&\text{ has }4(n-5)+7=4n-13\text{ tokens}
\end{align*}
Note that $\mathcal{D}_{n,9}$ is defined for $n\geq 5$, while $\mathcal{D}_{n,10}$, $\mathcal{D}_{n,11}$, and $\mathcal{D}_{n,12}$ are defined for $n\geq 6$, and $\mathcal{D}_{n,13}$ is defined for $n\geq 7$. 
Thus for $n\geq 7$, $\mathcal{A}_n$, $\mathcal{C}_n$ and the above constructions are sufficient.  For $n=6$, consider $\mathcal{A}_6$, $\mathcal{D}_{6,11}$, $\mathcal{D}_{6,10}$, $\mathcal{D}_{6,9}$, and $\mathcal{C}_6$.  For $n=5$, consider $\mathcal{A}_5$, $\mathcal{D}_{5,9}$, and $\mathcal{C}_5$.
\end{proof}
\def\sporadicscale{0.35}
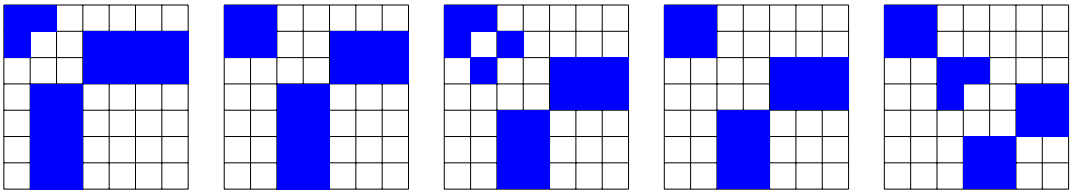
\begin{figure}[h]
    \centering
    \begin{tikzpicture}[scale=\sporadicscale] 
    \draw[step=1cm,black,very thin] (0,0) grid (7,7);
    \fill[blue] (1,0) rectangle (3,4);
    \fill[blue] (3,4) rectangle (7,6);
    \fill[blue] (0,5) rectangle (1,7);
    \fill[blue] (1,6) rectangle (2,7);
    \end{tikzpicture}
    \quad
    \begin{tikzpicture}[scale=\sporadicscale] 
    \draw[step=1cm,black,very thin] (0,0) grid (7,7);
    \fill[blue] (2,0) rectangle (4,4);
    \fill[blue] (4,4) rectangle (7,6);
    \fill[blue] (0,5) rectangle (2,7);
    \end{tikzpicture}
    \quad
    \begin{tikzpicture}[scale=\sporadicscale] 
    \draw[step=1cm,black,very thin] (0,0) grid (7,7);
    \fill[blue] (2,0) rectangle (4,3);
    \fill[blue] (4,3) rectangle (7,5);
    \fill[blue] (0,5) rectangle (1,7);
    \fill[blue] (1,6) rectangle (2,7);
    \fill[blue] (1,4) rectangle (2,5);
    \fill[blue] (2,5) rectangle (3,6);
    \end{tikzpicture}
    \quad
    \begin{tikzpicture}[scale=\sporadicscale] 
    \draw[step=1cm,black,very thin] (0,0) grid (7,7);
    \fill[blue] (2,0) rectangle (4,3);
    \fill[blue] (4,3) rectangle (7,5);
    \fill[blue] (0,5) rectangle (2,7);
    \end{tikzpicture}
    \quad
    \begin{tikzpicture}[scale=\sporadicscale] 
    \draw[step=1cm,black,very thin] (0,0) grid (7,7);
    \fill[blue] (3,0) rectangle (5,2);
    \fill[blue] (5,2) rectangle (7,4);
    \fill[blue] (0,5) rectangle (2,7);    
    \fill[blue] (2,3) rectangle (3,5);
    \fill[blue] (3,4) rectangle (4,5);
    \end{tikzpicture}
    \caption{The penults $\mathcal{D}_{7,9}$, $\mathcal{D}_{7,10}$, $\mathcal{D}_{7,11}$, $\mathcal{D}_{7,12}$, $\mathcal{D}_{7,13}$}
\label{fig:spradic penults for tic}
\end{figure}

\begin{cor*}
    The possible number of tokens in a penult on an $n\times n$ board of the dual game of \textsc{impartial tic} is
    \begin{center}
    \begin{tabular}{c|c}
        $n$ & set of possible number of tokens in a penult \\\hline
        $2$ & $\{4\}$ \\
        $3$ & $\{6\}$ \\
        $4$ & $\{8,9\}$ \\
        $\geq 5$ & $[2n,4(n-2)]$ 
    \end{tabular}    
    \end{center}    
\end{cor*}

\section{The Penults of \textsc{Impartial Tak}}\label{section - impartial tak}

We now turn our attention to the penults of \textsc{impartial tak}.  For $n\geq 1$, let $L(n)$ and $U(n)$ denote the size of the smallest and largest penult on an $n\times n$ board.  Clearly, each row of a penult must be missing at least two tokens.  Hence $U(n)\leq n^2-2n$.  On the other hand, this bound is sharp:
    

\begin{figure}[h]
    \centering
    \begin{tikzpicture}[scale=0.4] 
    \draw[step=1cm,black,very thin] (0,0) grid (7,7);
    \fill[blue] (0,6) rectangle (1,7);
    \fill[blue] (4,6) rectangle (7,7);
    \fill[blue] (0,0) rectangle (1,2);
    \fill[blue] (1,0) rectangle (2,1);
    \fill[blue] (6,0) rectangle (7,1);
    \fill[blue] (5,5) rectangle (6,6);
    \fill[blue] (6,4) rectangle (7,6);
    \fill[blue] (1,2) rectangle (2,6);
    \fill[blue] (2,1) rectangle (3,6);
    \fill[blue] (3,1) rectangle (4,6);
    \fill[blue] (4,1) rectangle (5,5);
    \fill[blue] (5,1) rectangle (6,4);
    \draw[red, thick] (1.5,6.5) -- (3.5,6.5);
    \draw[red, thick] (3.5,6.5) -- (6.5,3.5);
    \draw[red, thick] (0.5,5.5) -- (0.5,2.5);
    \draw[red, thick] (0.5,2.5) -- (2.5,0.5);
    \draw[red, thick] (3.5,0.5) -- (5.5,0.5);
    \draw[red, thick] (6.5,1.5) -- (6.5,2.5);
    \end{tikzpicture}
    \qquad
    \begin{tikzpicture}[scale=0.4] 
    \draw[step=1cm,black,very thin] (0,0) grid (5,5);
    \fill[blue] (0,0) rectangle (3,3);
    \fill[blue] (3,3) rectangle (4,4);
    \fill[blue] (4,4) rectangle (5,5);
    \end{tikzpicture}
    \qquad
    \begin{tikzpicture}[scale=0.4] 
    \draw[step=1cm,black,very thin] (0,0) grid (5,5);
    \fill[blue] (0,0) rectangle (3,3);
    \fill[blue] (4,3) rectangle (5,4);
    \fill[blue] (3,4) rectangle (5,5);
    \end{tikzpicture}
    \caption{(a) The Variable Diamond, enumeration emphasized. (b), (c) The L-Snakes}
    \label{fig:intermediate penults}
\end{figure}
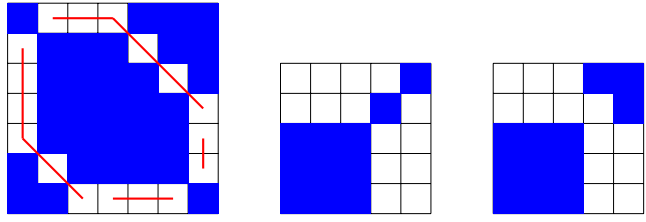

\begin{prop}
    Let $n\geq 4$ and $n^2-4(n-2)-2 \leq x \leq n^2-2n$.  Then there is a penult on an $n\times n$ board with $x$ tokens.  
\end{prop}

\begin{cor*}
For all $n\geq 2$, $U(n)= n^2-2n$.
\end{cor*}

\begin{proof}
    Consider the Variable Diamond penult in Figure \ref{fig:intermediate penults}.  Let the number of spaces without tokens in the top row and left column  be $k$ and $l$, respectively.  Then the number of tokens is
    \[n^2-2(n-1)-(k-1)-(l-1)=n^2-2n-k-l+4\]
    where $2\leq k,l\leq n-2$.  Thus the possible number of tokens is the interval $[n^2-4(n-2),n^2-2n]$.  
    Next note that the number of tokens in the L-Snake penults in \ref{fig:intermediate penults} are $(n-2)^2+2=n^2-4(n-2)-2$ and $(n-2)^2+3=n^2-4(n-2)-1$.
\end{proof}

We now turn our attention to the lower bound.

\begin{prop}\label{prop:snake}
    Suppose $n\geq 6$.  There is a penult on an $n\times n$ board with 
    \[\begin{cases}
    2n+(n+2)(\frac{n-4}{3})&=\frac{n^2}{3}+\frac{4}{3}n-\frac{8}{3}\text{ if }n= 1\bmod 6\\
    2(n-2)+n(\frac{n-2}{3})&=\frac{n^2}{3}+\frac{4}{3}n-4 \text{ if }n= 2\bmod 6\\
    (n-2)+(n-1)+(n+1)(\frac{n-3}{3})&=\frac{n^2}{3}+\frac{4}{3}n-4\text{ if }n= 3\bmod 6\\
    2n+(n+2)(\frac{n-4}{3})&=\frac{n^2}{3}+\frac{4}{3}n-\frac{8}{3}\text{ if }n= 4\bmod 6\\
    2n+(n-2)+(n+2)(\frac{n-5}{3})&=\frac{n^2}{3}+2n-\frac{16}{3} \text{ if }n= 5\bmod 6\\
    3(n-2)+n(\frac{n-3}{3})&=\frac{n^2}{3}+2n-6\text{ if }n= 0\bmod 6
    \end{cases}\] tokens.  Hence $L(n)\leq \frac{n^2}{3}+2n-\frac{8}{3}$, for all $n\geq 6$.
\end{prop}
\def\offset{0.2}
\def\snakescale{0.285}
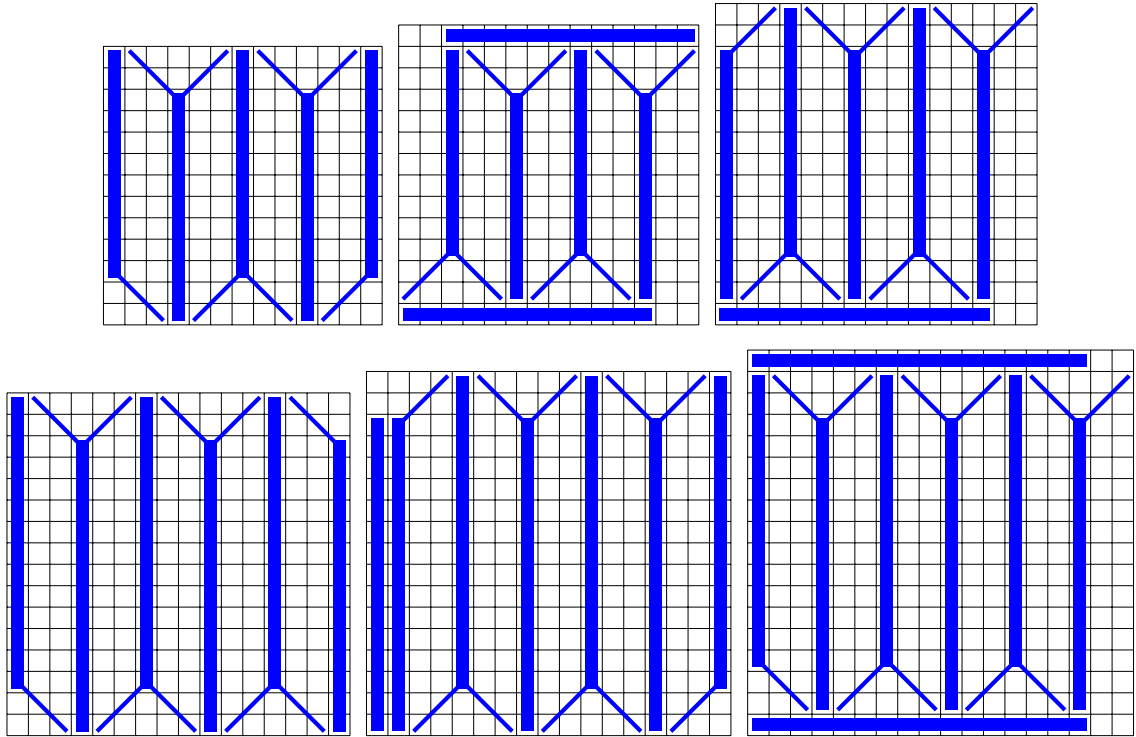
\begin{figure}[h]
    \centering

    \begin{tikzpicture}[scale=\snakescale] 
    \draw[step=1cm,black,very thin] (0,0) grid (13,13);
    \foreach \x in {0,6,12}
    {\fill[blue] (\x+\offset,2+\offset) rectangle (1+\x-\offset,13-\offset);}

    \foreach \x in {0,6}
    {\draw[blue, ultra thick] (0.5+\x,2.5) -- (0.5+\x+2.3,2.5-2.3);}

    \foreach \x in {6,12}
    {\draw[blue, ultra thick] (0.5+\x,2.5) -- (0.5+\x-2.3,2.5-2.3);}

    \foreach \x in {3,9}
    {\fill[blue] (\x+\offset,0+\offset) rectangle (1+\x-\offset,11-\offset);
    \draw[blue, ultra thick] (0.5+\x, 10.5) -- (0.5+\x+2.3, 10.5+2.3);
    \draw[blue, ultra thick] (0.5+\x, 10.5) -- (0.5+\x-2.3, 10.5+2.3);
    }

    \end{tikzpicture}\;
    \begin{tikzpicture}[scale=\snakescale] 
    \draw[step=1cm,black,very thin] (0,0) grid (14,14);
    \fill[blue] (0+\offset,0+\offset) rectangle (12-\offset,1-\offset);
    \fill[blue] (2+\offset,13+\offset) rectangle (14-\offset,14-\offset);

    \foreach \x in {5,11}
    {\fill[blue] (\x+\offset,1+\offset) rectangle (1+\x-\offset,11-\offset);
    \draw[blue, ultra thick] (0.5+\x, 10.5) -- (0.5+\x+2.3, 10.5+2.3);
    \draw[blue, ultra thick] (0.5+\x, 10.5) -- (0.5+\x-2.3, 10.5+2.3);}

    \foreach \x in {2,8}
    {\fill[blue] (\x+\offset,3+\offset) rectangle (1+\x-\offset,13-\offset);
    \draw[blue, ultra thick] (0.5+\x, 3.5) -- (0.5+\x+2.3, 3.5-2.3);
    \draw[blue, ultra thick] (0.5+\x, 3.5) -- (0.5+\x-2.3, 3.5-2.3);
    }

    \end{tikzpicture}\;
    \begin{tikzpicture}[scale=\snakescale] 
    \draw[step=1cm,black,very thin] (0,0) grid (15,15);
    \fill[blue] (0+\offset,0+\offset) rectangle (13-\offset,1-\offset);
    \foreach \x in {0,6,12}
    {\fill[blue] (\x+\offset,1+\offset) rectangle (1+\x-\offset,13-\offset);
    \draw[blue, ultra thick] (0.5+\x,12.5) -- (0.5+\x+2.3,12.5+2.3);}

    \foreach \x in {6,12}
    {
    \draw[blue, ultra thick] (0.5+\x,12.5) -- (0.5+\x-2.3,12.5+2.3);}

    \foreach \x in {3,9}
    {\fill[blue] (\x+\offset,3+\offset) rectangle (1+\x-\offset,15-\offset);
    \draw[blue, ultra thick] (0.5+\x,3.5) -- (0.5+\x+2.3,3.5-2.3);
    \draw[blue, ultra thick] (0.5+\x,3.5) -- (0.5+\x-2.3,3.5-2.3);}

    \end{tikzpicture}\bigskip
    
    \begin{tikzpicture}[scale=\snakescale] 
    \draw[step=1cm,black,very thin] (0,0) grid (16,16);
    \foreach \x in {3,9,15}
    {\fill[blue] (\x+\offset,0+\offset) rectangle (1+\x-\offset,14-\offset);
    \draw[blue, ultra thick] (0.5+\x,13.5) -- (0.5+\x-2.3,13.5+2.3);}

    \foreach \x in {3,9}
    {\draw[blue, ultra thick] (0.5+\x,13.5) -- (0.5+\x+2.3,13.5+2.3);}

    \foreach \x in {0,6,12}
    {\fill[blue] (\x+\offset,2+\offset) rectangle (1+\x-\offset,16-\offset);
    \draw[blue, ultra thick] (0.5+\x,2.5) -- (0.5+\x+2.3,2.5-2.3);}

    \foreach \x in {6,12}
    {\draw[blue, ultra thick] (0.5+\x,2.5) -- (0.5+\x-2.3,2.5-2.3);}

    \end{tikzpicture}\;
    \begin{tikzpicture}[scale=\snakescale] 
    \draw[step=1cm,black,very thin] (0,0) grid (17,17);
    \foreach \x in {0,1,7,13}
    {\fill[blue] (0+\x+\offset,0+\offset) rectangle (1+\x-\offset,15-\offset);}

    \foreach \x in {1,7,13}
    {\draw[blue, ultra thick] (\x+0.5, 14.5) -- (\x+0.5+2.3, 14.5+2.3);}

    \foreach \x in {7,13}
    {\draw[blue, ultra thick] (\x+0.5, 14.5) -- (\x+0.5-2.3, 14.5+2.3);}

    \foreach \x in {4,10,16}
    {\fill[blue] (0+\x+\offset,2+\offset) rectangle (1+\x-\offset,17-\offset);}

    \foreach \x in {4,10,16}
    {\draw[blue, ultra thick] (0.5+\x,2.5) -- (0.5+\x-2.3,2.5-2.3);}

    \foreach \x in {4,10}
    {\draw[blue, ultra thick] (0.5+\x,2.5) -- (0.5+\x+2.3,2.5-2.3);}
  
    \end{tikzpicture}\;
    \begin{tikzpicture}[scale=\snakescale] 
    \draw[step=1cm,black,very thin] (0,0) grid (18,18);
    \fill[blue] (0+\offset,0+\offset) rectangle (16-\offset,1-\offset);
    \fill[blue] (0+\offset,17+\offset) rectangle (16-\offset,18-\offset);
    \foreach \x in {0,6,12} 
        {\fill[blue] (3+\x+\offset,1+\offset) rectangle (4+\x-\offset,15-\offset);
        \draw[blue, ultra thick] (3.5+\x,14.5) -- (3.5+\x-2.3,14.5+2.3);
        \draw[blue, ultra thick] (3.5+\x,14.5) -- (3.5+\x+2.3,14.5+2.3);}
        
    \foreach \x in {-3,3,9} 
        {\fill[blue] (3+\x+\offset,3+\offset) rectangle (4+\x-\offset,17-\offset);
        \draw[blue, ultra thick] (3.5+\x,3.5) -- (3.5+\x+2.3,3.5-2.3);}

    \foreach \x in {3,9} 
        {\draw[blue, ultra thick] (3.5+\x,3.5) -- (3.5+\x-2.3,3.5-2.3);}
    \end{tikzpicture}
    \caption{The Snake Diagrams for $n=13,14,15,16,17,18$, drawn to emphasize the enumeration of the tokens.}
    \label{fig:snake diagram}
\end{figure}
\begin{proof}
    See the Snake Diagrams in Figure \ref{fig:snake diagram}.
\end{proof}
We are pretty sure the L-Snakes can be iteratively transformed into the Snake Diagrams via a sequence of penults, though we don't see a non-tedious way to verify this:

\begin{conj}
        For all $n\geq 6$, let $M_n$ be the value given by Proposition \ref{prop:snake}, and let $M_n\leq x\leq n^2-4(n-2)-2$.  Then there is a penult on an $n\times n$ board with $x$ tokens.
\end{conj}
We also think that the snake diagrams are essentially minimal:
\begin{conj}
        For all $n$, $L(n)\geq \frac{n^2}{3}+an+b$, for some constants $a$ and $b$.
\end{conj}
We next show some partial progress toward this latter conjecture.
\begin{prop}
    If $n\geq 5$, a penult on an $n\times n$ board must have at least 
$\frac{7}{26}(n-8)^2$
tokens.
\end{prop}
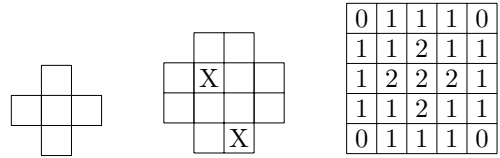
\begin{figure}[h]
    \centering
    \begin{tikzpicture}[scale=0.4] 
    \draw[step=1cm,black,very thin] (1,0) grid (2,3);
    \draw[step=1cm,black,very thin] (2,1) grid (3,2);
    \draw[step=1cm,black,very thin] (0,1) grid (1,2);
    \end{tikzpicture}
    \qquad
    \begin{tikzpicture}[scale=0.4] 
    \draw[step=1cm,black,very thin] (1,0) grid (3,4);
    \draw[step=1cm,black,very thin] (0,1) grid (1,3);
    \draw[step=1cm,black,very thin] (3,1) grid (4,3);
    \node at (2.5,0.5) {X};
    \node at (1.5,2.5) {X};
    \end{tikzpicture}
    \qquad
    \begin{tikzpicture}[scale=0.4] 
    \draw[step=1cm,black,very thin] (0,0) grid (5,5);
    \node at (1.5,2.5) {2};
    \node at (2.5,2.5) {2};
    \node at (3.5,2.5) {2};
    \node at (2.5,3.5) {2};
    \node at (2.5,1.5) {2};
    \node at (1.5,0.5) {1};
    \node at (2.5,0.5) {1};
    \node at (3.5,0.5) {1};
    \node at (0.5,1.5) {1};
    \node at (1.5,1.5) {1};
    \node at (3.5,1.5) {1};
    \node at (4.5,1.5) {1};
    \node at (0.5,2.5) {1};
    \node at (4.5,2.5) {1};
    \node at (0.5,3.5) {1};
    \node at (1.5,3.5) {1};
    \node at (3.5,3.5) {1};
    \node at (4.5,3.5) {1};
    \node at (1.5,4.5) {1};
    \node at (2.5,4.5) {1};
    \node at (3.5,4.5) {1};
    \node at (0.5,0.5) {0};
    \node at (4.5,0.5) {0};
    \node at (0.5,4.5) {0};
    \node at (4.5,4.5) {0};
    \end{tikzpicture}
    \caption{(a) The Cross, (b) the Thick Cross, and (c) the Weighted Cross}
    \label{fig:cross diagram}
\end{figure}
\begin{proof}
    Consider the diagrams in Figure \ref{fig:cross diagram}. Any penult must have at least one token in each Cross.  Indeed, if all five spaces are empty, then a token could be placed in the middle. 

    Similarly, any penult must have at least three tokens in each Thick Cross. 
    Indeed suppose a penult contains a Thick Cross with only two tokens.  If both tokens are in the boundary spaces, then there will be a sub-Cross with no tokens.  If both tokens are in the middle squares, then a third token can be placed in another middle square.  If one token is in a middle square and one in a boundary square, then to avoid an empty Cross, they are placed as in the figure.  In this case, a third token may be placed in any empty middle square.

    A similar case analysis shows that, in a penult, the sum of the weights of the tokens in any Weighted Cross is at least $7$.

    So consider any penult on an $n\times n$ board, and for each $1\leq i,j\leq n$, let $x_{i,j}=1$ if there is a token in position $(i,j)$ and $0$ if not.  Then $\sum_{1\leq i,j\leq n} x_{i,j}$ is the number of tokens in the penult.

    Meanwhile, we have for all $0\leq r,c\leq (n-4)$,
    \[x_{r+1,c+2} + x_{r+1,c+3} + x_{r+1,c+4} + x_{r+2,c+1} + x_{r+2,c+2} + 2x_{r+2,c+3} +
    \ldots + x_{r+5,c+4} \geq 7,\]
    where the coefficients correspond to those in the Weighted Cross.  Then summing over $0\leq r,c\leq (n-4)$, 
    we obtain
    \[\sum_{5\leq i,j\leq (n-4)} 26 x_{i,j} + \text{ edge terms} \geq 7(n-4)^2.\]
    Next, since the edge terms all have coefficients less than $26$, we have
    \[26*(\text{total number of tokens})= 26\sum_{1\leq i,j\leq n} x_{i,j}\geq \sum_{5\leq i,j\leq (n-4)} 26 x_{i,j} + \text{ edge terms}\geq 7(n-4)^2.\]
    In other words, the total number of tokens is at least $\frac{7}{26}(n-4)^2$
\end{proof}
\subsection*{Computation} We have computed all penults up to $6\times 6$ using Maple.  
As expected, $L(4)=6$. We then saw that $L(5)=10$ and $L(6)=16$, and that the possible numbers of tokens in each case are intervals.
There are a total of $59$ nonisometric $4\times 4$ penults, and $3629$ nonisometric $5\times 5$ penults.
Curiously, there is a \emph{unique} $5\times 5$ penult with $10$ tokens (see Figure \ref{fig:min 5x5 penult}), and this is the only size board and number of tokens for which we have observed this uniqueness!
See the last author's website, \url{http://www.thotsaporn.com}, for all relevant Maple code. 
\begin{figure}[h]
    \centering
    \begin{tikzpicture}[scale=0.4] 
    \draw[step=1cm,black,very thin] (0,0) grid (5,5);
    \fill[blue] (0,0) rectangle (1,2);
    \fill[blue] (2,0) rectangle (3,1);
    \fill[blue] (4,0) rectangle (5,1);
    \fill[blue] (3,2) rectangle (4,5);
    \fill[blue] (0,4) rectangle (1,5);
    \fill[blue] (1,3) rectangle (2,5);
    \end{tikzpicture}
    \caption{The unique minimal $5\times 5$ penult}
    \label{fig:min 5x5 penult}
\end{figure}
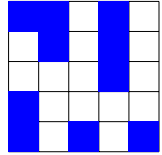
\section{Winning strategies for \textsc{Impartial Tic} and \textsc{Impartial Tak}}\label{section - winning strategies}
As usual, let $n$ denote the side length of the game board. For both games under consideration, the winning strategy depends on the parity of $n$.
\begin{prop} 
In \textsc{impartial tic}, if $n$ is even, second player wins via a mirroring strategy across the origin.
If $n$ is odd, first player wins by first taking the center square, and then mirroring across the origin.  
In both cases, the winning player may need to deviate from this strategy for their winning move.
\end{prop}
\begin{proof}
    If $n$ is even, the second player's moves will always result in a board position that is symmetric about the origin.  Furthermore, the second player always moves in a different row and column than the preceding move. Hence, it will be the first player who first creates a row or column with $n-1$ tokens.  Then second player then can take the win.

    If $n$ is odd, the proof is similar, switching the roles of the first and second player, except that we must consider the middle row and middle column separately.  Each of these will have an even number of spaces left open after each move by the first player.  Thus the second player cannot win by completing one of these.
\end{proof}
 It is easy to check that mirroring across a center line will also work for the even case of \textsc{impartial tic}, though not for the odd case.
\begin{prop}\label{prop:strategy for even tak}
For \textsc{impartial tak}, if $n$ is even, second player wins via a mirroring strategy across a center line.   They may need to deviate from this strategy for the winning move. 
\end{prop}
\begin{proof}
    Suppose $n$ is even, the second player is mirroring across a vertical center line, and that the first player has just won with a move on the left half.  If the winning path connects the left and right edges, then the portion of the winning path on the right half, together with its mirror image, must already have been a winning path, a contradiction.  
    
    If the winning path connects the top and bottom edges, then first note that there was also a winning path that did not cross the center line.  If we call the first player's winning move $(a,b)$, then the second player could have won on their previous move by instead moving at either $(a,b)$ or $(n-a,b)$, whichever is on the side with more tokens.  
\end{proof}
Mirroring across the origin will not be a winning strategy for the even case of \textsc{impartial tak}.  In fact, neither of these strategies, nor mirroring across a diagonal will work for the general odd case.  See Figure \ref{fig:wrong winning strategies} for counterexamples.  While we suspect first player wins for all odd $n$, we can only prove it for $n=5$.
\begin{figure}[h]
    \centering
    \begin{tikzpicture}[scale=0.4] 
    \draw[step=1cm,black,very thin] (0,0) grid (4,4);
    \fill[blue] (1,0) rectangle (2,2);
    \fill[blue] (2,2) rectangle (3,4);
    \end{tikzpicture}
    \qquad
    \begin{tikzpicture}[scale=0.4] 
    \draw[step=1cm,black,very thin] (0,0) grid (5,5);
    \fill[blue] (0,0) rectangle (1,2);
    \fill[blue] (2,2) rectangle (3,3);
    \fill[blue] (0,3) rectangle (1,5);
    \draw[red, thick, dotted] (0,2.5) -- (5,2.5);
    \end{tikzpicture}
    \qquad
    \begin{tikzpicture}[scale=0.4] 
    \draw[step=1cm,black,very thin] (0,0) grid (7,7);
    \fill[blue] (3,3) rectangle (4,4);
    \fill[blue] (1,1) rectangle (6,2);
    \fill[blue] (2,0) rectangle (3,1);
    \fill[blue] (1,2) rectangle (2,3);
    \fill[blue] (5,2) rectangle (6,3);
    \fill[blue] (5,4) rectangle (6,5);
    \fill[blue] (1,4) rectangle (2,5);
    \fill[blue] (1,5) rectangle (6,6);
    \fill[blue] (4,6) rectangle (5,7); 
    \node[red] at (2.5,0.5) {X};
    \node[red] at (4.5,6.5) {X};
    \end{tikzpicture}
    \qquad
    \begin{tikzpicture}[scale=0.4] 
    \draw[step=1cm,black,very thin] (0,0) grid (7,7);
    \fill[blue] (3,3) rectangle (4,4);
    \fill[blue] (0,0) rectangle (1,2);
    \fill[blue] (1,1) rectangle (5,2);
    \fill[blue] (5,2) rectangle (6,7);
    \fill[blue] (6,6) rectangle (7,7);
    \draw[red, thick, dotted] (0,7) -- (7,0);
    \end{tikzpicture}
    \caption{In (a), second player loses when mirroring across the origin on an even board.  In (b), (c), (d), first player loses on an odd board when mirroring across a center line, the origin, and a diagonal, respectively.  Where noted, the tokens marked with an X were the last two moves.}
    \label{fig:wrong winning strategies}
\end{figure}
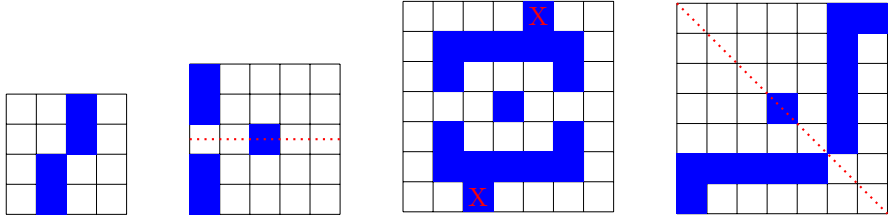
\begin{prop}\label{prop:strategy for odd tak}
For \textsc{impartial tak}, if $n=5$ is odd,  first player wins by first taking the center square, and then mirroring across the origin.   They may need to deviate from this strategy for the winning move.  
\end{prop}
\begin{proof}
Suppose $n=5$, the first player is following this strategy, and the second player has just won with the move $(a,b)$.  

If there is a winning path which contains the center square, then there is also one symmetric about the origin.  This one does not contain $(a,b)$, a contradiction.

If there is a winning path disconnected from its mirror image, then the first player could have won on their previous move by instead moving at either $(a,b)$ or $(n-a,n-b)$, whichever makes a connected component with more tokens.

The remaining case is that there is a winning path, disconnected from the middle, and connected to its mirror image.  On a $5\times 5$ board, this means that if we also put a token on $(n-a,n-b)$, then all four edges will be connected by a single connected component which encircles the middle.  Thus deleting the tokens at $(a,b)$ and $(n-a,n-b)$ will leave opposite edges connected, a contradiction. 
\end{proof}

\begin{conj}
    If $n>5$ is odd, then the first player has a winning strategy for \textsc{impartial tak}.
\end{conj}

\section{Further adventures}\label{section - further adventures}
\subsection*{Penult Intervals. \textsc{D\&B}}
Is there an natural example of a impartial, normal-play, positional game in which the possible number of tokens in a penult on a fixed board size does \emph{not} form an interval?  At first glance, \textsc{D\&B} also seems to follow the same pattern as \textsc{impartial tic} and \textsc{impartial tak}.  Representative penults for the smallest cases are shown in Figure \ref{fig:d&b 3x3}.
\def\offsetA{0.15}
\begin{figure}[h]
    \centering
    \begin{tikzpicture}[scale=0.4] 
    \foreach \x in {0,1,2}
    \foreach \y in {0,1,2}
    {\node[black] at (\x,\y) {$\bullet$};}
    \draw[red, thick] (0+\offsetA,1) -- (1-\offsetA,1);
    \draw[red, thick] (1+\offsetA,1) -- (2-\offsetA,1);
    \draw[red, thick] (1,0+\offsetA) -- (1,1-\offsetA);
    \draw[red, thick] (1,1+\offsetA) -- (1,2-\offsetA);
    \end{tikzpicture}
    \qquad
    \begin{tikzpicture}[scale=0.4] 
    \foreach \x in {0,1,2}
    \foreach \y in {0,1,2}
    {\node[black] at (\x,\y) {$\bullet$};}
    \draw[red, thick] (0+\offsetA,1) -- (1-\offsetA,1);
    \draw[red, thick] (1+\offsetA,1) -- (2-\offsetA,1);
    \draw[red, thick] (1,0+\offsetA) -- (1,1-\offsetA);
    \draw[red, thick] (2-\offsetA,2) -- (1+\offsetA,2);
    \draw[red, thick] (0+\offsetA,2) -- (1-\offsetA,2);
    \end{tikzpicture}
    \qquad
    \begin{tikzpicture}[scale=0.4] 
    \foreach \x in {0,1,2}
    \foreach \y in {0,1,2}
    {\node[black] at (\x,\y) {$\bullet$};}
    \draw[red, thick] (0+\offsetA,1) -- (1-\offsetA,1);
    \draw[red, thick] (1+\offsetA,1) -- (2-\offsetA,1);
    \draw[red, thick] (0+\offsetA,0) -- (1-\offsetA,0);
    \draw[red, thick] (2-\offsetA,0) -- (1+\offsetA,0);
    \draw[red, thick] (2-\offsetA,2) -- (1+\offsetA,2);
    \draw[red, thick] (0+\offsetA,2) -- (1-\offsetA,2);
    \end{tikzpicture}
    \qquad
    \begin{tikzpicture}[scale=0.4] 
    \foreach \x in {0,1,2}
    \foreach \y in {0,1,2}
    {\node[black] at (\x,\y) {$\bullet$};}
    \draw[red, thick] (0+\offsetA,1) -- (1-\offsetA,1);
    \draw[red, thick] (2,0+\offsetA) -- (2,1-\offsetA);
    \draw[red, thick] (2,2-\offsetA) -- (2,1+\offsetA);
    \draw[red, thick] (0+\offsetA,0) -- (1-\offsetA,0);
    \draw[red, thick] (2-\offsetA,0) -- (1+\offsetA,0);
    \draw[red, thick] (2-\offsetA,2) -- (1+\offsetA,2);
    \draw[red, thick] (0+\offsetA,2) -- (1-\offsetA,2);
    \end{tikzpicture}
    \qquad
    \begin{tikzpicture}[scale=0.4] 
    \foreach \x in {0,1,2}
    \foreach \y in {0,1,2}
    {\node[black] at (\x,\y) {$\bullet$};}
    \draw[red, thick] (0,0+\offsetA) -- (0,1-\offsetA);
    \draw[red, thick] (0,2-\offsetA) -- (0,1+\offsetA);
    \draw[red, thick] (2,0+\offsetA) -- (2,1-\offsetA);
    \draw[red, thick] (2,2-\offsetA) -- (2,1+\offsetA);
    \draw[red, thick] (0+\offsetA,0) -- (1-\offsetA,0);
    \draw[red, thick] (2-\offsetA,0) -- (1+\offsetA,0);
    \draw[red, thick] (2-\offsetA,2) -- (1+\offsetA,2);
    \draw[red, thick] (0+\offsetA,2) -- (1-\offsetA,2);
    \end{tikzpicture}
    \newline    
    \newline    
     \begin{tikzpicture}[scale=0.4] 
    \foreach \x in {0,1,2,3}
    \foreach \y in {0,1,2,3}
    {\node[black] at (\x,\y) {$\bullet$};}
    \draw[red, thick] (0+\offsetA,1) -- (1-\offsetA,1);
    \draw[red, thick] (1,0+\offsetA) -- (1,1-\offsetA);
    \draw[red, thick] (2,0+\offsetA) -- (2,1-\offsetA);
    \draw[red, thick] (2+\offsetA,1) -- (3-\offsetA,1);
    \draw[red, thick] (0+\offsetA,2) -- (1-\offsetA,2);
    \draw[red, thick] (1,2+\offsetA) -- (1,3-\offsetA);
    \draw[red, thick] (2,2+\offsetA) -- (2,3-\offsetA);
    \draw[red, thick] (2+\offsetA,2) -- (3-\offsetA,2);
    \end{tikzpicture}
    \quad    
     \begin{tikzpicture}[scale=0.4] 
    \foreach \x in {0,1,2,3}
    \foreach \y in {0,1,2,3}
    {\node[black] at (\x,\y) {$\bullet$};}
    \draw[red, thick] (0+\offsetA,1) -- (1-\offsetA,1);
    \draw[red, thick] (1,0+\offsetA) -- (1,1-\offsetA);
    \draw[red, thick] (2,0+\offsetA) -- (2,1-\offsetA);
    \draw[red, thick] (2+\offsetA,1) -- (3-\offsetA,1);
    \draw[red, thick] (0+\offsetA,2) -- (1-\offsetA,2);
    \draw[red, thick] (1,2+\offsetA) -- (1,3-\offsetA);
    \draw[red, thick] (2,2+\offsetA) -- (2,3-\offsetA);
    \draw[red, thick] (3,2+\offsetA) -- (3,3-\offsetA);
    \draw[red, thick] (3,1+\offsetA) -- (3,2-\offsetA);
    \end{tikzpicture}
    \quad    
     \begin{tikzpicture}[scale=0.4] 
    \foreach \x in {0,1,2,3}
    \foreach \y in {0,1,2,3}
    {\node[black] at (\x,\y) {$\bullet$};}
    \draw[red, thick] (0+\offsetA,1) -- (1-\offsetA,1);
    \draw[red, thick] (1,0+\offsetA) -- (1,1-\offsetA);
    \draw[red, thick] (2,0+\offsetA) -- (2,1-\offsetA);
    \draw[red, thick] (3,0+\offsetA) -- (3,1-\offsetA);
    \draw[red, thick] (2,1+\offsetA) -- (2,2-\offsetA);
    \draw[red, thick] (0+\offsetA,2) -- (1-\offsetA,2);
    \draw[red, thick] (1,2+\offsetA) -- (1,3-\offsetA);
    \draw[red, thick] (2,2+\offsetA) -- (2,3-\offsetA);
    \draw[red, thick] (3,2+\offsetA) -- (3,3-\offsetA);
    \draw[red, thick] (3,1+\offsetA) -- (3,2-\offsetA);
    \end{tikzpicture}
    \quad    
     \begin{tikzpicture}[scale=0.4] 
    \foreach \x in {0,1,2,3}
    \foreach \y in {0,1,2,3}
    {\node[black] at (\x,\y) {$\bullet$};}
    \draw[red, thick] (1,1+\offsetA) -- (1,2-\offsetA);
    \draw[red, thick] (0,0+\offsetA) -- (0,1-\offsetA);
    \draw[red, thick] (1,0+\offsetA) -- (1,1-\offsetA);
    \draw[red, thick] (2,0+\offsetA) -- (2,1-\offsetA);
    \draw[red, thick] (3,0+\offsetA) -- (3,1-\offsetA);
    \draw[red, thick] (2,1+\offsetA) -- (2,2-\offsetA);
    \draw[red, thick] (0+\offsetA,2) -- (1-\offsetA,2);
    \draw[red, thick] (1,2+\offsetA) -- (1,3-\offsetA);
    \draw[red, thick] (2,2+\offsetA) -- (2,3-\offsetA);
    \draw[red, thick] (3,2+\offsetA) -- (3,3-\offsetA);
    \draw[red, thick] (3,1+\offsetA) -- (3,2-\offsetA);
    \end{tikzpicture}
    \quad    
     \begin{tikzpicture}[scale=0.4] 
    \foreach \x in {0,1,2,3}
    \foreach \y in {0,1,2,3}
    {\node[black] at (\x,\y) {$\bullet$};}
    \draw[red, thick] (1,1+\offsetA) -- (1,2-\offsetA);
    \draw[red, thick] (0,0+\offsetA) -- (0,1-\offsetA);
    \draw[red, thick] (1,0+\offsetA) -- (1,1-\offsetA);
    \draw[red, thick] (2,0+\offsetA) -- (2,1-\offsetA);
    \draw[red, thick] (3,0+\offsetA) -- (3,1-\offsetA);
    \draw[red, thick] (2,1+\offsetA) -- (2,2-\offsetA);
    \draw[red, thick] (0,2+\offsetA) -- (0,3-\offsetA);
    \draw[red, thick] (0,1+\offsetA) -- (0,2-\offsetA);
    \draw[red, thick] (1,2+\offsetA) -- (1,3-\offsetA);
    \draw[red, thick] (2,2+\offsetA) -- (2,3-\offsetA);
    \draw[red, thick] (3,2+\offsetA) -- (3,3-\offsetA);
    \draw[red, thick] (3,1+\offsetA) -- (3,2-\offsetA);
    \end{tikzpicture}
    \quad    
     \begin{tikzpicture}[scale=0.4] 
    \foreach \x in {0,1,2,3}
    \foreach \y in {0,1,2,3}
    {\node[black] at (\x,\y) {$\bullet$};}
    \draw[red, thick] (1,1+\offsetA) -- (1,2-\offsetA);
    \draw[red, thick] (0,0+\offsetA) -- (0,1-\offsetA);
    \draw[red, thick] (0+\offsetA,0) -- (1-\offsetA,0);
    \draw[red, thick] (1+\offsetA,0) -- (2-\offsetA,0);
    \draw[red, thick] (2+\offsetA,0) -- (3-\offsetA,0);
    \draw[red, thick] (3,0+\offsetA) -- (3,1-\offsetA);
    \draw[red, thick] (2,1+\offsetA) -- (2,2-\offsetA);
    \draw[red, thick] (0,2+\offsetA) -- (0,3-\offsetA);
    \draw[red, thick] (0,1+\offsetA) -- (0,2-\offsetA);
    \draw[red, thick] (1,2+\offsetA) -- (1,3-\offsetA);
    \draw[red, thick] (2,2+\offsetA) -- (2,3-\offsetA);
    \draw[red, thick] (3,2+\offsetA) -- (3,3-\offsetA);
    \draw[red, thick] (3,1+\offsetA) -- (3,2-\offsetA);
    \end{tikzpicture}
    \quad    
     \begin{tikzpicture}[scale=0.4] 
    \foreach \x in {0,1,2,3}
    \foreach \y in {0,1,2,3}
    {\node[black] at (\x,\y) {$\bullet$};}
    \draw[red, thick] (1,1+\offsetA) -- (1,2-\offsetA);
    \draw[red, thick] (0,0+\offsetA) -- (0,1-\offsetA);
    \draw[red, thick] (0+\offsetA,0) -- (1-\offsetA,0);
    \draw[red, thick] (1+\offsetA,0) -- (2-\offsetA,0);
    \draw[red, thick] (2+\offsetA,0) -- (3-\offsetA,0);
    \draw[red, thick] (3,0+\offsetA) -- (3,1-\offsetA);
    \draw[red, thick] (2,1+\offsetA) -- (2,2-\offsetA);
    \draw[red, thick] (0,2+\offsetA) -- (0,3-\offsetA);
    \draw[red, thick] (0,1+\offsetA) -- (0,2-\offsetA);
    \draw[red, thick] (0+\offsetA,3) -- (1-\offsetA,3);
    \draw[red, thick] (1+\offsetA,3) -- (2-\offsetA,3);
    \draw[red, thick] (2+\offsetA,3) -- (3-\offsetA,3);
    \draw[red, thick] (3,2+\offsetA) -- (3,3-\offsetA);
    \draw[red, thick] (3,1+\offsetA) -- (3,2-\offsetA);
    \end{tikzpicture}
    \quad   
    \caption{A penult for \textsc{D\&B} on a $3\times 3$ board can have $4$, $5$, $6$, $7$, or $8$ tokens. A penult for \textsc{D\&B} on a $4\times 4$ board can have from $8$ to $14$ tokens.}
    \label{fig:d&b 3x3}
\end{figure}
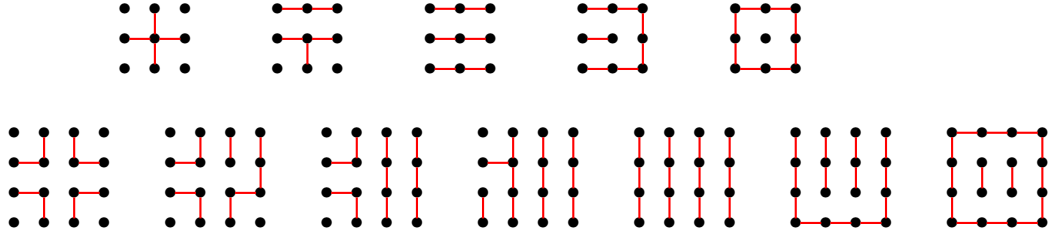
\subsection*{Mate-in-$n$}  Positions of impartial, normal-play combinatorial games are first partitioned into $\mathcal{W}$- and $\mathcal{L}$-positions.  The $\mathcal{W}$-positions are further partitioned by \textsc{nim}-values.  It is natural to partition $\mathcal{L}$-positions as follows:
\begin{itemize}
    \item If the position is a terminal position, call it \emph{mate-in-$0$}
    \item Let $\alpha>0$ be an ordinal.  Suppose 
    \begin{itemize}
        \item each option itself has an option that is mate-in-$\gamma$ for some $\gamma<\alpha$; and
        \item the position is not mate-in-$\beta$ for any $\beta<\alpha$,
    \end{itemize} 
    then call it \emph{mate-in-$\alpha$}
\end{itemize}
Then penults are precisely those positions which are mate-in-$1$.  For finite $k$, a position is mate-in-$k$ if $k$ is the largest number such that there is an option whose only winning response is a position that is mate-in-$(k-1)$.  In other words, it is the largest $k$ for which the losing player can force the game to need $2k$ moves to terminate. Examples:
\begin{itemize}
    \item In $\textsc{subtract}\{1,2,3\}$, the position $4n$ is mate-in-$n$.
    \item In \textsc{wythoff's game}, the position $(\lfloor k\phi\rfloor, \lfloor k\phi\rfloor+k)$ is mate-in-$k$.  In this case, the only winning response to the move $(\lfloor (k-1)\phi\rfloor, \lfloor k\phi\rfloor+k)$ is $(\lfloor (k-1)\phi\rfloor, \lfloor (k-1)\phi\rfloor+k-1)$.
    \item A losing position of \textsc{nim} with $2n$ total tokens is mate-in-$n$.  If the first player takes one token from the pile with the smallest non-zero binary digit, the second player will be forced to take a single token from another pile with that same digit.
\end{itemize}

Notice that this concept does not coincide with either \emph{birthday}, as all of these game positions are of \textsc{nim}-value $0$, nor with \emph{formal birthday}, as evidenced by the example $\{\{\emptyset,\{\emptyset\}\}\}$.  This game position has formal birthday $3$ and is mate-in-$1$, that is, a penult.
As it is such a natural notion in \textsc{chess}, it is curious that mate-in-$n$ appears little, if at all, in the CGT literature.

\subsection*{Partizan games. Non-positional games}  Definition \ref{definition of penult} is not restricted to  positional games.  It can also be extended to partizan games as follows.
\begin{defi}
    Let $G=\{L\mid R\}$ be a game position.  
    \begin{itemize}
        \item $G$ is a \textit{left-penult} if $R\neq\emptyset$, and for all $\{A\mid B\}\in R$, $\emptyset\in A$.
        \item $G$ is an \textit{right-penult} if $L\neq\emptyset$, and for all $\{A\mid B\}\in L$, $\emptyset\in B$.
        \item $G$ is a \textit{penult} if it is both a left-penult and a right-penult.
    \end{itemize}
\end{defi}
The penults of \cite{Quarto}, for example, might be worth a look.

\subsection*{Mis\`{e}re play}   Is there a reasonable notion of a penult for mis\`{e}re-play games?  For example, \cite{Notakto} is a positional game canonically played in mis\`{e}re.

\end{document}